\documentstyle[12pt,amstex,amssymb,amsthm]{article}

\newcommand{\C}{\mbox{\rm \,l\kern-0.52em C}}
\newcommand{\Ce}{\rm \,l\kern-0.35em C}
\newcommand{\R}{{\rm l}\!{\rm R}}
\newcommand{\Z}{{\sf Z}\!\!{\sf Z}}

\newtheorem{theorem}{Theorem}[section]

\newtheorem{prop}[theorem]{Proposition}

\newtheorem{cor}[theorem]{Corollary}

\newtheorem{lemma}[theorem]{Lemma}

\renewenvironment{proof}{{\bf Proof:}}{\mbox{}\hfill $\Box$}
\newenvironment{proof21}{{\bf Proof of Theorem 3.1:}}{\mbox{}\hfill $\Box$}
\newenvironment{proof22}{{\bf Proof of Theorem 3.2:}}{\mbox{}\hfill $\Box$}
\newenvironment{proof23}{{\bf Proof of Theorem 3.3:}}{\mbox{}\hfill $\Box$}
\theoremstyle{definition}
\newtheorem{definition}[theorem]{Definition}
\title{Finitely summable Fredholm modules over higher rank groups and
lattices}
\author{Michael Puschnigg}
\date{}
\begin{document}
\maketitle
\section{Introduction}
In this note we use recent progress in rigidity theory to adress the question of existence of finitely summable Fredholm modules over  
group $C^*$-algebras. 

We begin by giving some motivation and recalling previous work on this problem.

According to Kasparov \cite{Ka1}, the $K$-homology groups of a $C^*$-algebra $A$ may be defined as the group of homotopy classes of bounded Fredholm modules $A$. In \cite{Co1} Connes constructed a Chern-character for Fredholm modules which satisfy a strong regularity condition, called finite summability (with respect to a dense subalgebra of $A$). This Chern-character is given by a completely explicit cyclic cocycle. In particular, the index pairing of the given module with $K$-theory can be calculated by a simple index formula. 

It appears therefore natural to ask whether every $K$-homology class can be represented by a finitely summable module.

In the case of the $C^*$-algebra of continuous functions on a smooth compact manifold the answer is affirmative. Every $K$-homology class may be represented by a module which is finitely summable over the algebra of smooth functions and Connes' character formula boils down to the classical Atiyah-Singer index formula in de Rham cohomology.

The situation changes if one passes to noncommutative $C^*$-algebras.

Recall that $K$-homology of $C^*$-algebras can be described either in terms of bounded or in terms of unbounded Fredholm modules.
In \cite{Co3} Connes has shown that there are no finitely summable unbounded Fredholm modules over the reduced $C^*$-algebra of any nonamenable discrete group at all! For bounded modules the story is different. In \cite{Co1} Connes constructs bounded Fredholm modules over the reduced group $C^*$-algebras of free groups and lattices in $SU(1,1)$ which are finitely summable over the group ring while representing nontrivial classes in $K$-homology. (Note that these groups are not amenable.)

The main observation in this note is that such examples do not exist 
if one passes to lattices in higher rank Lie groups. In fact we have

\begin{theorem}
\begin{itemize}
\item[i)] Let $G$ be a product of simple real Lie groups of real rank at least 2. Then every finitely summable Fredholm representation (see section 1.2) of $G$ is homotopic to a multiple of thr trivial representation.

\item[ii)] Let $\Gamma$ be a lattice in a product of simple real Lie groups of real rank at least two. Then every Fredholm module over $C^*_{max}(\Gamma)$, which is finitely summable over $\C\Gamma$, is homotopic to a finite dimensional virtual representation of $\Gamma$.
\end{itemize}
\end{theorem}

As a consequence we obtain in section 2 a complete description of the 
classes in $KK_G(\C,\C)$, $KK_*(C^*_{max}(\Gamma),\C)$ and $KK_*(C^*_{red}(\Gamma),\C)$ which can be realized by finitely summable Fredholm representations or Fredholm modules. 

In particular, no nontrivial class in $KK_*(C^*_{red}(\Gamma),\C)$, $\Gamma$ a higher rank lattice,  can be represented by a finitely summable module. This shows that for noncommutative $C^*$-algebras it is in general not possible to realise every $K$-homology class by a finitely summable Fredholm module. 

The previous results are a simple consequence of recent deep work by Bader, Furman, Gelander and Monod in rigidity theory. They conjecture that every equicontinuous isometric action of a higher rank group or lattice on a uniformly convex Banach space admits a global fixed point. They verify this conjecture for isometric actions on classical $L^p$-spaces.
One of their main results reads as follows: 

\begin{theorem}{\bf (Bader, Furman, Gelander, Monod)} \cite{BFGM}
Let $G$ be a higher rank group or a higher rank lattice as in Theorem 1.1. Let $(X,\mu)$ be a standard measure space and let $\rho$ be an isometric linear representation of $G$ on $L^p(X,\mu)$. 
Then $$H^1_{cont}(G,\rho)\,=\,0$$ for $1<p<\infty$. 
\end{theorem}

Adapted to noncommutative $\ell^p$-spaces their result immediately implies our main theorem. The link is given by the well known correspondence between group cohomology and the Hochschild cohomology of group algebras, as was already observed and exploited by Connes in \cite{Co3}.

After developing the notions of finitely summable $K$-homology and finitely\\ summable Fredholm representations in section 2 we are ready to state our results in section 3. Section 4 recalls the parts of the work of Bader, Furman, Gelander and Monod needed for our purpose and in section 5  we finally give the proofs of the theorems stated in section 3.

I want to express my gratitude to Thierry Fack, Fuad Kittaneh, Hideki Kosaki and Antony Wassermann for helpful discussions and correspondence concerning eigenvalue estimates for compact operators and in particular   Hideki Kosaki for bringing Ando's work \cite{A} to my attention.

\section{The setup}
\subsection{Smooth $K$-homology}

The $K$-homology groups of a $C^*$-algebra $A$ are defined, according to
Kasparov \cite{Ka1}, as the groups of homotopy classes of Fredholm
modules  over $A$. 

Recall that an even {\bf Fredholm module} ${\cal E}\,=\,({\cal
H}^\pm,\,\rho^\pm,\,F)$ over $A$ is given by 
 a $\Z/2\Z$-graded complex Hilbert space $\cal H=\cal H^\pm$, a
pair $\rho=\rho^\pm:A\,\to\,{\cal L(H)}^\pm$ of representations of $A$,
and an odd bounded linear operator $F:\,\cal H^\pm\,\to\,\cal H^\mp$ satisfying
$$
\begin{array}{ccc}
\rho(A)(F^2-1)\subset{\cal
K(H)}, & \rho(A)(F^*-F)\subset{\cal K(H)}, & [F,\rho(A)]\subset{\cal K(H)}.\\
\end{array}
\eqno(2.1) 
$$

Here ${\cal
K(H)}$ denotes the ideal of compact operators in ${\cal
L(H)}$.
The notion of an odd Fredholm representation is
obtained by forgetting the $\Z/2\Z$-grading. 

Fredholm modules should be viewed as generalized elliptic operators over $A$.
In fact, if $M$ is a smooth compact manifold without boundary, then every
linear elliptic differential operator over $M$ gives rise to a Fredholm module
over $A=C(M)$, the algebra of continuous complex-valued functions on $M$. It
can be shown that such geometric modules generate the whole $K$-homology
group $KK_*(C(M),\C)$ of $M$. 

The previous example suggests a notion of smoothness (finitely summability) for Fredholm modules.

\begin{definition}

Let $A$ be a $C^*$-algebra and let ${\cal E}\,=\,({\cal
H}^\pm,\,\rho^\pm,\,F)$ be a Fredholm module over $A$. Fix $1\leq p < \infty$.
Then $\cal E$ is said to be {\bf p-summable} over the dense (involutive) subalgebra ${\cal
A}\subset A$ if instead of (2.1) the stronger conditions 
$$
\begin{array}{ccc}
\rho({\cal A})(F^2-1)\subset\ell^p({\cal H}), & \rho({\cal A})(F^2-1)\subset\ell^p({\cal H}), & [F,\rho({\cal
A})]\subset\ell^p({\cal H}).\\ 
\end{array}
\eqno(2.2) 
$$
hold where
$$
\ell^p({\cal H})\,=\,\{T\in{\cal K(H)},\,Trace(T^*T)^{\frac{p}{2}}<\infty\,\}
\eqno(2.3) 
$$
denotes the Schatten ideal of $p$-summable compact operators. 
\end{definition}

According to Weyl, the Fredholm module associated to an elliptic operator $\cal
D$ over the smooth compact manifold $M$ is $p$-summable over ${\cal
C}^\infty(M)\subset C(M)$ for\\ $p>\frac{dim(M)}{ord({\cal D})}$.

Kasparov's $K$-homology groups $KK(A,\C)$ are defined as the groups of
equivalence classes of even (odd) Fredholm modules over $A$ with respect to the
equivalence relation generated by unitary equivalence, addition of degenerate
modules, and operator homotopy.

Following Nistor \cite{Ni}, we introduce similar relations for smooth modules.

\begin{definition}
Let $A$ be a $C^*$-algebra and let ${\cal A}\subset A$ be a dense (involutive)
subalgebra. Denote by ${\cal
E}_*^{(p)}((A,{\cal A}),\C)$ the set of even(odd) Fredholm modules over $A$
which are $p$-summable over $\cal A$. 
\end{definition}

Recall that a Fredholm module is called degenerate if all the terms in (2.1)
are identically zero. Every degenerate module over $A$ is thus 
$p$-summable over any dense subalgebra so that the set ${\cal
E}_*^{(p)}((A,{\cal A}),\C)$ is stable under
addition of degenerate modules. There is also an obvious notion of unitary
equivalence among smooth Fredholm modules. One has to be a bit careful about
the appropriate notion of operator homotopy in the smooth context.

\begin{definition}
A smooth operator homotopy between ${\cal E}_0=({\cal H},\rho,F_0)$ and ${\cal
E}_1=({\cal H},\rho,F_1)$, ${\cal E}_0,{\cal E}_1\in{\cal
E}_*^{(p)}((A,{\cal A}),\C)$, is given by a family $F_t,\,t\in[0,1]$
of bounded operators on $\cal H$ connecting $F_0$ and $F_1$ such that
\begin{itemize}
\item ${\cal E}_t=({\cal H},\rho,F_t)\in{\cal
E}_*^{(p)}((A,{\cal A}),\C),\,0\leq t\leq 1$,
\item $\{t\mapsto F_t\}\in{\cal C}^\infty([0,1],{\cal L(H)})$,
\item the maps $a\mapsto \rho(a)(F_t^2-1),\,\,a\mapsto \rho(a)(F_t^*-F_t),$\\

 and $a\mapsto [F_t,\rho(a)]$ lie in ${\cal L}({\cal A},{\cal
C}^\infty([0,1],\ell^p({\cal H}))$.
\end{itemize}
In the last condition $\cal A$ is assumed to be equipped with a bornology
 turning the inclusion $\cal A\subset A$ into a bounded operator. 
Thus bounded linear maps from $\cal A$ to bornological (or Fr\'echet) spaces
make sense. In our examples we will  only deal with the fine bornology
generated by finite subsets, so that the boundedness is automatic.  
\end{definition}

\begin{definition}
Let $1\leq p < \infty$ and let $\cal A$ be a dense (involutive) subalgebra of
the $C^*$-algebra $A$. The smooth ($p$-summable) $K$-homology groups 
$$
\begin{array}{ccc}
{\cal KK}_*^{(p)}((A,{\cal A}),\C) & = & {\cal E}_*^{(p)}((A,{\cal
A}),\C)/\sim \\ 
\end{array}
$$
of the pair $(A,{\cal A})$ are defined as the groups of equivalence classes of
$p$-summable Fredholm modules over $(A,{\cal A})$  with respect to the equivalence
relation generated by unitary equivalence, addition of degenerate modules, and
smooth operator homotopy. 
\end{definition}

There is an obvious homomorphism
$$
\begin{array}{ccc}
{\cal KK}_*^{(p)}((A,{\cal A}),\C) & \longrightarrow & KK_*(A,\C) \\  
\end{array}
\eqno(2.4)
$$
One is mainly interested in the image of this map.

If a $K$-homology class can be represented by a smooth
$K$-cycle, then its Chern-Connes character can be given by a simple explicit
character formula. This makes index calculations for smooth $K$-cycles much
easier than in the general case.

\begin{prop}
Let $\cal A$ be a dense (involutive) subalgebra of the separable $C^*$-algebra
$A$. Let  
$$
\begin{array}{cccc}
\check{ch}^{(p)}: & {\cal KK}_*^{(p)}((A,{\cal A}),\C) & \longrightarrow &
HP^*({\cal A}) \\   
\end{array}
\eqno(2.5)
$$
be Connes' Chern character on smooth $K$-homology with values in
periodic cyclic cohomology \cite{Co1} and let
$$
\begin{array}{cccc}
\check{ch}: & KK_*(A,\C) & \longrightarrow &
HC^*_{loc}(A) \\   
\end{array}
\eqno(2.6)
$$
be the Chern-Connes character on $K$-homology with values in local cyclic
cohomology \cite{Pu}. Then there is a natural commutative diagram
$$
\begin{array}{ccccc}
{\cal KK}_*^{(p)}((A,{\cal A}),\C) & & \longrightarrow & & KK_*(A,\C) \\
 & & & & \\
\check{ch}^{(p)}\downarrow & & & & \downarrow\check{ch} \\
 & & & & \\
HP^*({\cal A}) & \longrightarrow & HC^*_{loc}({\cal A}) & \longleftarrow &
HC^*_{loc}(A) \\
\end{array}
\eqno(2.7)
$$
\end{prop} 

\subsection{Representation rings of locally compact groups}

Let $G$ be a locally compact (second countable) group. We denote by $R(G)$ the
ring of unitary equivalence classes of finite dimensional unitary
representations of $G$. Sum and product in $R(G)$ are given by the direct sum
and the tensor product of (equivalence classes) of representations.

A remarkable generalization of the notion of finite dimensional unitary
representation has been introduced by G.~Kasparov \cite{Ka2}:

\begin{definition}
A {\bf Fredholm representation} ${\cal E}_G\,=\,({\cal H},\rho, F)$ of $G$ is
given by  
\begin{itemize}
\item a $\Z/2\Z$-graded Hilbert space ${\cal H}^\pm$
\item a pair of unitary representations $\rho^\pm:\,G\longrightarrow\,{\cal
U}({\cal H}^\pm)$ of $G$ on the  even(odd) part of $\cal H$.
\item an odd bounded operator $F:{\cal H}^\pm\longrightarrow{\cal H}^\mp$,
which almost intertwines the representations $\rho^+$ and $\rho^-$, i.e.
$$
\begin{array}{c}
\{g\mapsto \rho(g)F\rho(g)^{-1}\}\,\in\,C(G,{\cal L(H)}),  \\
  \\
F^2-1\in{\cal K(H)},\,(F-F^*)\in{\cal K(H)}, 
  \,\,\rho(g)F\rho(g)^{-1}\,-\,F \in{\cal K(H)},\,\forall g\in G. \\
\end{array}
$$
\end{itemize}
\end{definition}

We propose the following notion of a smooth Fredholm representation

\begin{definition}
A Fredholm representation ${\cal E}_G\,=\,({\cal H},\rho, F)$ of a locally
compact group $G$ is {\bf p-summable}  if 
$$
\begin{array}{ccc}
F^2-1\in\ell^p({\cal H}), & F^*-F\in\ell^p({\cal H}), & \{g\mapsto \rho(g)F\rho(g)^{-1}\,-\,F\}\,\in\,C(G,\ell^p({\cal H})).\\
\end{array}
$$
\end{definition}

Denote by $Fred(G)$ and $Fred^{(p)}(G)$, respectively, the set of
Fredholm representations and $p$-summable Fredholm representations of $G$, respectively. 

If $G=\Gamma$ is a discrete group, there is a tautological bijection
$$
\begin{array}{ccc}
Fred^{(p)}(\Gamma) & \overset{\simeq}{\longrightarrow} & 
{\cal E}_0^{(p)}((C^*_{max}(\Gamma),\C\Gamma),\C) \\
\end{array}
\eqno(2.8)
$$
between the set of $p$-summable Fredholm representations of $\Gamma$ and the
set of even Fredholm modules over the enveloping $C^*$-algebra
$C^*_{max}(\Gamma)$ of $\ell^1(\Gamma)$, which are $p$-summable over
$\C\Gamma$.

There are obvious notions of unitary equivalence, degeneracy, and operator
homotopy of Fredholm representations. For finitely summable representations 
one puts
\begin{definition}
A smooth operator homotopy between $p$-summable Fredholm representations 
${\cal E}_0=({\cal H},\rho,F_0)$ and ${\cal E}_1=({\cal H},\rho,F_1)$ 
is given by a family $F_t,\,t\in[0,1]$
of bounded operators on $\cal H$ connecting $F_0$ and $F_1$ such that
\begin{itemize}
\item ${\cal E}_t=({\cal H},\rho,F_t)\in 
Fred^{(p)}(G),\,0\leq t\leq 1$,
\item $\{t\mapsto F_t\}\in{\cal C}^\infty([0,1],{\cal L(H)})$,
\item $\{g\mapsto \rho(g)\,F_t\,\rho(g)^{-1}\,-\,F_t\}\in 
C(G,{\cal C}^\infty([0,1],\ell^p{\cal(H)}))$.
\end{itemize}
\end{definition}

Kasparov defines the $G$-equivariant $K$-homology of $\C$ as
$$
\begin{array}{ccc}
KK_G(\C,\C) & = & Fred(G)/\sim \\
\end{array}
\eqno(2.9)
$$
as the group of equivalence classes of Fredholm representations of $G$ with
respect to the equivalence relation generated by unitary equivalence, addition
of degenerate representations, and operator homotopy.

The smooth version of this is 
\begin{definition}
Let $G$ be a locally compact group and let $1\leq p<\infty$. The $p$-summable,
$G$-equivariant $K$-homology of $\C$ 
$$
\begin{array}{ccc}
{\cal KK}_G^{(p)}(\C,\C) & = & Fred^{(p)}(G)/\sim \\
\end{array}
\eqno(2.10)
$$
is defined as the group of equivalence classes of $p$-summable Fredholm
representations of $G$ with respect to the equivalence relation generated by
unitary equivalence, addition of degenerate representations, and smooth
operator homotopy.
\end{definition}

If $G=\Gamma$ is discrete there are tautological isomorphisms
$$
\begin{array}{ccc}
KK_\Gamma(\C,\C) & \overset{\simeq}{\longrightarrow} & KK(C^*_{max}(\Gamma),\C)\\ 
\end{array}
\eqno(2.11)
$$
and
$$
\begin{array}{ccc}
{\cal KK}_{\Gamma}^{(p)}(\C,\C) & \overset{\simeq}{\longrightarrow} & 
{\cal KK}^{(p)}((C^*_{max}(\Gamma),\C\Gamma),\C)\\
\end{array}
\eqno(2.12)
$$

The Kasparov product turns the (smooth) equivariant $K$-groups (2.9) and
(2.10) into unital, associative rings. The tautological isomorphisms 
(2.11) and (2.12) become than ring isomorphisms.

Finally there is a natural ring homomorphism
$$
\begin{array}{ccc}
R(G) & \longrightarrow & 
KK_G(\C,\C)\\
\end{array}
\eqno(2.13)
$$
from the representation ring of $G$ to the $G$-equivariant $K$-homology 
of $\C$, which assigns to a virtual finite dimensional unitary representation 
$[\rho_+]-[\rho_-]$ the Fredholm representation ${\cal
E}_G\,=\,(\varrho_+\oplus\varrho_-,0)$. This homomorphism factors through the
smooth equivariant $K$-homology rings ${\cal KK}_G^{(p)}(\C,\C)$ for all $1\leq
p<\infty$.

\section{Results} 
In the sequel we understand by a {\bf higher rank group} the group 
of $k$-rational points of a Zariski-connected, simple, isotropic algebraic
group of $k$-rank at least two over the local field $k$.

\begin{theorem}
Let $G$ be a finite product of higher rank groups. Then
$$
\begin{array}{ccc}
{\cal KK}_G^{(p)}(\C,\C) & \simeq & \Z \\
\end{array}
\eqno(3.1)
$$
is generated by the trivial representation of $G$ for $1<p<\infty$.
\end{theorem}

\begin{theorem}
Let $\Gamma$ be a lattice (i.e. a discrete subgroup of finite
covolume) in a finite product of higher rank groups. 
Then the tautological homomorphisms
$$
\begin{array}{ccccc}
R(\Gamma) & \overset{\simeq}{\longrightarrow} & {\cal KK}_\Gamma^{(p)}(\C,\C)
&  \overset{\simeq}{\longrightarrow} & {\cal
KK}^{(p)}_0((C^*_{max}(\Gamma),\C\Gamma),\C) \\
\end{array}
\eqno(3.2)
$$
are isomorphisms and 
$$
\begin{array}{ccc}
{\cal KK}^{(p)}_1((C^*_{max}(\Gamma),\C\Gamma),\C) & = & 0\\
\end{array}
\eqno(3.3)
$$
for $1<p<\infty$.
\end{theorem}

\begin{theorem}
Let $\Gamma$ be a lattice in a finite product of higher rank groups.
 Then
$$
\begin{array}{ccc}
{\cal KK}^{(p)}_*((C^*_{red}(\Gamma),\C\Gamma),\C) & = & 0 \\
\end{array}
\eqno(3.4)
$$
for $1<p<\infty$, where $C^*_{red}(\Gamma)$ denotes the reduced $C^*$-algebra
of $\Gamma$, i.e. the enveloping $C^*$-algebra of the regular representation of
$\Gamma$. 
\end{theorem}

\begin{cor}
Under the assumptions of 3.3 no nonzero class in the $K$-homology group $KK_*(C^*_{red}(\Gamma),\C)$ 
can be represented by a Fredholm module which is finitely summable over $\C\Gamma$.
\end{cor}

\section{The work of Bader, Furman, Gelander and Monod} 

\subsection{Uniformly convex Banach spaces}

Recall that a Banach space $B$ is {\bf strictly convex} if the midpoint of
every segment joining two points of the unit sphere $S(B)$ is in the interior
of the unit ball. It is {\bf uniformly convex} if the previous condition holds
uniformly in the following sense: for all $\delta>0$ there exists $\epsilon>0$
such that 
$$
\begin{array}{ccc}
\parallel x-y \parallel\,\geq\,\delta & \Rightarrow & 
\parallel\frac{x+y}{2}\parallel\,\leq\,1-\epsilon \\
\end{array}
\eqno(4.1)
$$
for all $x,y\in S(B)$.

If $B$ is uniformly convex, then for every $\xi\in S(B)$ there is a unique $x\in
S(B^*)$ such that $\langle x,\xi\rangle=1$. This assignement defines a uniformly
continuous homeomorphism 
$$
\iota:\,S(B)\longrightarrow S(B^*)
\eqno(4.2)
$$
from the unit sphere of
$B$ to that of $B^*$. Uniformly convex spaces are reflexive. A Banach space is
uniformly smooth if its dual space is uniformly convex.

Let $\rho:G\to {\bf O}(B)$ be an orthogonal (i.e. isometric)
linear representation on a strictly convex real Banach space $B$. 
Then the subspace $B^{\rho(G)}$ of $G$-fixed vectors possesses a canonical
complement  
$$
B'\,=\,\left((B^*)^{\rho_*(G)}\right)^{\ominus}
\eqno(4.3)
$$ 
given by the subspace annihilated by the space $(B^*)^{\rho_*(G)}$ of all
bounded linear functionals on $B$ which are fixed under the contragredient
representation $\rho^*$ of $G$.

The representation $\rho$ splits as direct sum of its restriction to an
orthogonal representation on $B'$ and the trivial representation on
$B^{\rho(G)}$.
$$
B\,\simeq\,B^{\rho(G)}\,\oplus\,B'
\eqno(4.4)
$$ 
\subsection{Isometric actions on uniformly convex spaces}

An isometric action of a locally group $G$ on a Banach space $B$ is called
continuous if the map $\pi:G\times B\longrightarrow B$ is continuous. Every
isometry of a strictly convex space is real affine (it preserves
midpoints of segments). Thus 
$$
\pi(g)v\,=\,\rho(g)v\,+\,\psi(g),\,\,\,\forall g\in G,\,\forall v\in B
\eqno(4.5)
$$
where $\rho:G\longrightarrow{\bf O}(B)$ is a continuous orthogonal
($\R$-linear isometric) representation of $G$ on $B$ and
$\psi:G\longrightarrow B$ is a continuous 1-cocycle on $G$ with values in
the $G$-module $B$: 
$$
\psi\,\in\,Z^1_{cont}(G,\rho)\,=\,\{\psi\in
C(G,B),\,\psi(gh)\,=\,\psi(g)\,+\,\rho(g)\psi(h),\,\forall g,h\in G\}. 
\eqno(4.6)
$$
There is thus a bijection between affine isometric $G$-actions on $B$ with
linear part $\rho$ and continuous 1-cocycles on $G$ with coefficients in
$\rho$. The topology of uniform convergence on compacta turns
$Z^1_{cont}(G,\rho)$ into a Fr\'echet space.

An affine isometric action (with linear part $\rho$) possesses a fixed point
if and only if the corresponding cocycle belongs to the subspace
$$
B^1(G,\rho)\,=\,\{\mu\in Z^1_{cont}(G,\rho),\,\mu(g)\,=\,\xi\,-\,\rho(g)\xi\,\,
\text{for some}\,\,\xi\in B\}. 
\eqno(4.7)
$$
The vanishing of the cohomology group 
$$
H^1_{cont}(G,\rho)\,=\,Z^1_{cont}(G,\rho)\,/B^1(G,\rho)
\eqno(4.8)
$$
is therefore equivalent to the assertion that every continuous isometric
action of $G$ on $B$ with linear part $\rho$ possesses a fixed point.

\begin{lemma}
Suppose that $H^1_{cont}(G,\rho)\,=\,0$ and let $\psi_t,t\in[0,1]$ be a smooth family of cocycles. Then there exists a family $\xi_t\in B,\,t\in[0,1]$ of fixed points of the corresponding affine actions which depends smoothly on $t$.
\end{lemma}
\begin{proof}
There is a canonical exact sequence 
$$
\begin{array}{cccccccccc}
0 & \to & B^{\rho(G)} & \to & B & \overset{p}{\to} & B^1(G,\rho) & \to & 0 \\
\end{array}
\eqno(4.9)
$$
of abstract vector spaces where $p$ assigns to a vector $\xi\in B$ the cocycle\\ $p(\xi):\,g\mapsto\xi-\rho(g)\xi$. Under our hypothesis $B^1(G,\rho)\,=\,Z^1_{cont}(G,\rho)$ and is thus complete. Therefore (4.9) becomes an exact sequence of Fr\'echet spaces which splits according to (4.4). In particular, the projection $p$ possesses a bounded linear section which assigns to every cocycle a fixed point of the corresponding affine action. This implies the lemma because bounded linear maps are smooth. 
\end{proof}

The result of Bader, Furman, Gelander and Monod that we wish to extend to
noncommutative $\ell^p$-spaces is the

\begin{theorem}\cite{BFGM}
Let $G$ be a product of higher rank groups or a lattice in such a group (see section 3). Let $(X,\mu)$ be a
measure space and let $\rho$ be any orthogonal representation of $G$ on the
Banach space $L^p(X,\mu)$ where $1<p<\infty$. Then $H^1_{cont}(G,\rho)=0$.
\end{theorem}

It should be mentioned that Bader, Furman, Gelander and Monod conjecture 
that any affine isometric action of a higher rank group or lattice on a
uniformly convex Banach space has a fixed point. In fact they deduce
the previous theorem from the following result which has a more general
flavor.

\begin{theorem}\cite{BFGM}
Let $G$ be a product of higher rank groups (see section 2) and let $\rho_B$
be an orthogonal representation of $G$ on a uniformly convex and uniformly smooth Banach space $B$.
Suppose that there exists an orthogonal representation $\rho$ of $G$ on a
Hilbert space $\cal H$ and a homeomorphism $\Phi:S(B)\longrightarrow S({\cal
H})$ between the unit spheres of $B$ and $\cal H$ such that
\begin{itemize}
\item $\Phi$ intertwines the $G$-actions induced by $\rho_B$ and $\rho$.
\item $\Phi$ and $\Phi^{-1}$ are uniformly continuous.
\end{itemize}
 Then $$H^1_{cont}(G,\rho_B)=0.$$
\end{theorem}

We will show in the sequel that this theorem applies to various affine
isometric actions of higher rank groups on operator ideals.

\section{Proofs}
\subsection{Geometry of Schatten ideals}
We fix from now on a real number $p,\,1<p<\infty$. The Schatten-von Neumann
ideal $\ell^p({\cal H})\subset{\cal L(H)}$ is the ideal of compact operators
whose sequence of singular values is $p$-summable. It is a Banach space with
respect to the norm 
$$\parallel T\parallel_p\,=\,Trace((T^*T)^{\frac{p}{2}})
\eqno(5.1)
$$
The Schatten spaces are symmetrically normed operator ideals, i.e. one has
$$
\parallel ATB\parallel_p\,\leq\,
\parallel A\parallel_{{\cal L(H)}}\parallel
T\parallel_p\parallel B\parallel_{{\cal L(H)}}
\eqno(5.2)
$$ 
for $A,B\in{\cal L(H)},\,T\in\ell^p({\cal H})$. In particular, if
$\rho:G\longrightarrow{\cal U(H)}$ is a unitary representation of a locally
compact group $G$ on $\cal H$, then 
$$
\rho_p:\,G \,\longrightarrow \, {\cal U}(\ell^p({\cal H})),\,\,
\rho_p(g)(T)\,=\,\rho(g)T\rho(g)^{-1}
\eqno(5.3)
$$
 is a strongly continuous, isometric, $\C$-linear representation of $G$ on
$\ell^p({\cal H})$. 

We will need rather precise information about the geometry of the Banach
spaces $\ell^p({\cal H})$ for $1<p<\infty$. We begin with

\begin{theorem}{\bf (Clarkson-MacCarthy),\cite{Si}}\\
The Banach spaces $\ell^p({\cal H}),\,1<p<\infty$ are uniformly convex.
\end{theorem}

This generalizes the correponding result for the classical Banach spaces
$L^p(X,\mu)$. From \cite{LT}, 1.e.9.(i) one derives further

\begin{cor}
Let $(X,\mu)$ be a Borel space. Then for $1< p<\infty$ the Banach space
$$
L^p(X,\ell^p({\cal H}))
\eqno(5.4)
$$
obtained from the space $C_c(X,\ell^p({\cal H}))$ of continuous compactly supported functions on $X$ with values in $\ell^p({\cal H})$ by completion with respect to the norm
$$
\parallel f\parallel_p^p\,\,\,=\underset{X}{\int}Trace((f^*(t)f(t))^{\frac{p}{2}})\,dt
\eqno(5.5)
$$ 
is uniformly convex.
\end{cor}

\begin{lemma}
Let $1<p,q<\infty$ be such that $\frac{1}{p}+\frac{1}{q}=1$. Then
$$
\begin{array}{lr}
\ell^p({\cal H})^*\,\simeq\,\ell^q({\cal H}), & L^p(X,\ell^p({\cal H}))^*\,\simeq\,L^q(X,\ell^q({\cal H})) \\
\end{array}
\eqno(5.6)
$$
in the notations of 5.2.
\end{lemma}

In particular, the Schatten ideals $\ell^p({\cal H}),\,1<p<\infty$ are
uniformly convex and uniformly smooth Banach spaces.

Now we will study the polar decomposition in $\ell^p({\cal H})$. The basic
result in this direction is 

\begin{theorem}{\bf (Powers-Stoermer inequality),\cite{PS}}\\
Let $A,B$ be positive trace class operators on $\cal H$. Then
$$
\begin{array}{lcr}
\parallel A^{\frac12} - B^{\frac12}\parallel_2^2 & \leq \,\,\,\parallel
A-B\parallel_1. \\ 
\end{array}
\eqno(5.7)
$$
\end{theorem}

This inequality has been generalized by various authors. In particular, 
Ando proved the following

\begin{theorem}{\bf (Generalized Powers-Stoermer inequality),\cite{A}}\\

Let $0<\alpha<1$ and let $A,B$ be positive operators in $\ell^p({\cal
H}),\,1\leq p<\infty$. 
Then 
$$ \begin{array}{lcr}
\parallel A^{\alpha} - B^{\alpha}\parallel_{\frac{p}{\alpha}}^{\frac{p}{\alpha}} & \leq
\,\,\,\parallel A-B\parallel_p^p. \\ 
\end{array}
\eqno(5.8)
$$
\end{theorem}

The generalized Powers-Stoermer inequalities are needed to establish the 

\begin{cor}
Let $1<p,q<\infty$. Then the Mazur map
$$ 
\begin{array}{cccc}
M_{p,q}: & S(\ell^p({\cal H})) & \longrightarrow & S(\ell^q({\cal H})) \\
 & T & \mapsto & sign(T)\vert T\vert^{\frac{p}{q}} \\ 
\end{array}
\eqno(5.9)
$$
is a uniformly continuous homeomorphism.
\end{cor}

\begin{proof}
Note that $M_{p,p}=Id$ and $M_{q,r}\circ M_{p,q}=M_{p,r}$ for $1<p,q,r<\infty$.
Therefore it suffices to verify the claim in the following two cases:
$1<q<p<3q,\,p\geq 2,$ and $p<q,\,\frac{1}{p}+\frac{1}{q}=1$.
In the second case it is easily seen that 
$M_{p,q}\,=\,*\circ\iota$ where $\iota$ is the canonical homeomorphism (4.2)
and $*:\ell^q({\cal H})\to\ell^q({\cal H})$ maps an operator to its adjoint.
Thus $M_{p,q}$ is oviously uniformly continuous in this case. If
$1<q<p<3q,\,p\geq 2,$ we find with $0<r=\frac12(\frac{p}{q}-1)<1$ that
$$
M_{p,q}(T)\,=\,sgn(T)\vert T\vert^{\frac{p}{q}}\,=\,
T\vert T\vert^{2r}\,=\,T(T^*T)^r
$$
The Hoelder and the generalized Powers-Stoermer inequalities yield then 
for $S,T\in S(\ell^p({\cal H}))$
$$
\parallel M_{p,q}(S)-M_{p,q}(T)\parallel_q\,\leq
$$
$$
\leq\,\parallel(S-T)(S^*S)^r\parallel_q\,+\,\parallel
T((S^*S)^r-(T^*T)^r)\parallel_q 
$$
$$
\leq\,\parallel
S-T\parallel_p\parallel(S^*S)^r\parallel_{\frac{p}{2r}}\,+\,\parallel
T\parallel_p\parallel(S^*S)^r-(T^*T)^r\parallel_{\frac{p}{2r}} $$ $$
\leq\,\parallel S-T\parallel_p(\parallel
S^*S\parallel_{\frac{p}{2}})^r\,+\,\parallel
T\parallel_p(\parallel S^*S-T^*T\parallel_{\frac{p}{2}})^r
$$
$$
\leq\,\parallel S-T\parallel_p(\parallel
S^*\parallel_p\parallel S\parallel_p)^r\,+\,\parallel
T\parallel_p(\parallel
S^*-T^*\parallel_p\parallel
S\parallel_p\,+\,\parallel
T^*\parallel_p\parallel S -T\parallel_p)^r 
$$
$$
\leq\,\parallel S-T\parallel_p\,+\,(2\parallel S-T\parallel_p)^r
$$
which shows that $M_{p,q}$ is uniformly continuous.
\end{proof}

Similar estimates lead to 

 \begin{cor}
Let $(X,\mu)$ be a Borel space and let $1<p,q<\infty$. Then the Mazur map
$$ 
\begin{array}{cccc}
M_{p,q}: & S(L^p(X,\ell^p({\cal H}))) & \longrightarrow & S(L^q(X,\ell^q({\cal H}))) \\
 & f & \mapsto & sign(f)\vert f\vert^{\frac{p}{q}} \\ 
\end{array}
\eqno(5.10)
$$
is a uniformly continuous homeomorphism.
\end{cor}

\subsection{Rigidity of actions on Schatten ideals}

We are ready to apply the results of Bader, Furman, Gelander and Monod in our framework.

\begin{theorem}
Let $G$ be a product of higher rank groups as in section 3 and let $\rho$ be a unitary representation of $G$ on the Hilbert space $\cal H$. Let $\rho_p$ be the corresponding isometric representation on the Schatten ideals $\ell^p({\cal H})$. Then 
$$
H^1_{cont}(G,\rho_p)\,=\,0\,\,\,\text{for}\,\,\,1<p<\infty.
\eqno(5.11)
$$
\end{theorem}

Consider now a lattice $\Gamma\subset G$ in a higher rank group. The homogeneous space $G/\Gamma$ is then of finite volume. According to 
5.2 and 5.3 the Banach spaces $L^p(G/\Gamma,\,\ell^p({\cal H}))$ are uniformly convex and uniformly smooth for $1<p<\infty$. Suppose that 
in addition a unitary representation $\rho$ of $\Gamma$ on a Hilbert space $\cal H$ is given and let $\rho_p$ be the corresponding isometric representation on $\ell^p({\cal H})$. Then $L^p(G/\Gamma,\,\ell^p({\cal H}))$ can be identified with the completion of the space $L^{[p]}(G,\,\ell^p({\cal H}))^\Gamma$ of $\Gamma$-equivariant maps from $G$ to $\ell^p({\cal H})$ with respect to the norm 
$\parallel f\parallel^p\,=\,\int_{{\cal D}}\parallel f(g)\parallel^p dg$ (${\cal D}\subset G$ a Borel fundamental domain for $\Gamma$). The latter space carries a canonical isometric linear $G$-action coming from left translation. 
 
\begin{theorem}
Let $G$ be a higher rank group and let $\Gamma$ be a lattice in $G$. Let $\rho$ be a unitary representation of $\Gamma$. Then in the previous notations
$$
H^1_{cont}(G,L^{[p]}(G,\,\ell^p({\cal H}))^\Gamma)\,=\,0\,\,\,\text{for}\,\,\,1<p<\infty.
\eqno(5.12)
$$
\end{theorem}

\begin{cor}
Let $\Gamma$ be a lattice in a product of higher rank groups as in section 2 and let $\rho$ be a unitary representation of $\Gamma$ on the Hilbert space $\cal H$. Let $\rho_p$ be the corresponding isometric representation on the Schatten ideals $\ell^p({\cal H})$. Then 
$$
H^1(\Gamma,\rho_p)\,=\,0\,\,\,\text{for}\,\,\,1<p<\infty.
\eqno(5.13)
$$
\end{cor}

\begin{proof} (of the theorem)
We begin with a proof of (5.11). We want to apply theorem 4.3 to the isometric representations $\rho_p$ on the Schatten ideals $\ell^p({\cal H})$. For $1<p<\infty$ these ideals are uniformly convex and uniformly smooth 5.3. Note that $\ell^2({\cal H})$ is the Hilbert space of Hilbert-Schmidt operators on $\cal H$. By construction the Mazur map 
$M_{p,2}:\,S(\ell^p({\cal H}))\,\to\,S(\ell^2({\cal H}))$ (5.10) intertwines the isometric representations $\rho_p$ and $\rho_2$. Moreover $M_{p,2}$ as well as its inverse $M_{2,p}$ are uniformly continuous by 5.6. Thus Theorem 4.3 of Bader, Furman, Gelander and Monod applies and yields (5.11). The demonstration of (5.12) is similar and uses 5.2 and 5.7.
\end{proof}

\begin{proof} (of the corollary)
The well known idea is to use induction of representations and affine isometric actions from $\Gamma$ to $G$ in order to reduce the statement to (5.12). The arguments in \cite{BFGM}, section 8.1, 8.2 apply verbatim to our situation and yield 
$H^1(\Gamma,\rho_p)\,\simeq\,H^1_{cont}(G,L^{[p]}(G,\,\ell^p({\cal H}))^\Gamma)$
The claim follows then from 5.9.
\end{proof}

\subsection{Finitely summable modules over higher rank groups and lattices}

The following consequence of the preceding rigidity theorems will immediately lead to our main results

\begin{prop}
Let $G$ be a product of higher rank groups or a lattice in such a group. Let $1<p<\infty$. 
\begin{itemize}
\item[i)] Every $p$-summable Fredholm representation ${\cal E}\,=\,({\cal H},\rho,F)$ is smoothly operator homotopic (among $p$-summable representations) to a Fredholm representation ${\cal E}'\,=\,({\cal H},\rho,F')$ such that $[F',\rho(G)]=0$.
\item[ii)] Let ${\cal E}_i\,=\,({\cal H},\rho,F_i),\,i=0,1,$ be $p$-summable Fredholm representations which are smoothly operator homotopic and satisfy $[F_i,\rho(G)]=0$. Then there exists a smooth operator homotopy ${\cal E}_t\,=\,({\cal H},\rho,F_t),\,t\in[0,1],$ connecting ${\cal E}_0$ and ${\cal E}_1$ and satisfying $[F_t,\rho(G)]=0,\,\forall t\in[0,1].$
\end{itemize}
\end{prop}

\begin{proof}
Let ${\cal E}\,=\,({\cal H},\rho,F)$ be a $p$-summable Fredholm representation. Then the map $\psi:G\to\ell^p({\cal H}),\,
g\mapsto \rho(g)F\rho(g)^{-1}\,-\,F$ defines a continuous one-cocycle 
$\psi\in Z^1_{cont}(G,\rho_p)$. According to theorems 5.8, 5.9, this cocycle is a coboundary, i.e. $\psi(g)\,=\,\rho(g)F\rho(g)^{-1}\,-\,F\,=\,
\rho(g)T\rho(g)^{-1}\,-\,T$ for some operator $T\in\ell^p({\cal H})$ and all $g\in G$. Thus ${\cal E}_s\,=\,({\cal H},\rho,F-sT)$ is a smooth operator homotopy
between ${\cal E}$ and the desired Fredholm representation ${\cal E}'\,=\,({\cal H},\rho,F-T)$ satisfying $[F-T,\rho(G)]=0$.
If ${\cal E}_t\,=\,({\cal H},\rho,F_t),\,t\in[0,1]$ is a smooth operator homotopy then the corresponding family of cocycles is smooth. According to 4.1 there exists a smooth family $T_t\in\ell^p({\cal H}),\,t\in[0,1]$ 
satisfying $[F_t,\rho(g)]=[T_t,\rho(g)],\,\forall g\in G, t\in[0,1]$. Thus ${\cal E}'_t\,=\,({\cal H},\rho,F_t-T_t),\,t\in[0,1]$ defines the desired operator homotopy consisting of operators commuting strictly with $G$.
\end{proof}

For the proof of 3.1 we still need the well known

\begin{theorem} 
\cite{BFGM} {\bf (Moore Ergodicity Theorem)}\\
Let $G(k)$ be the group of $k$-rational points of a simply connected, 
semisimple algebraic group over the local field $k$. Let $\rho$ be a unitary representation of $G$ on a Hilbert space $\cal{H}$. Then either there are nonzero vectors fixed under $\rho(G(k))$ or all matrix coefficients of $\rho$ vanish at infinity, i.e.
$\{ g\mapsto\langle\xi,\rho(g)\eta\rangle\}\in C_0(G),\,\forall\xi, \eta\in{\cal H}$.
\end{theorem}

\begin{proof21} 
Let ${\cal E}\,=\,({\cal H},\rho,F)$ be a $p$-summable Fredholm representation of $G$. After a smooth operator homotopy we may suppose by 5.10.i) that $F$ and $\rho(G)$ strictly commute. Moreover we may suppose that $F$ is selfadjoint. The subspace $Ker(F)\subset{\cal H}$ as well as its orthogonal complement $\cal H'$ are then stable under $\rho(G)$. The triple ${\cal E}'\,=\,({\cal H}',\rho_{\vert{\cal H}'}, F'=F_{\vert{\cal H}'})$ is then smoothly operator homotopic to the degenerate triple ${\cal D}'\,=\,({\cal H}',\rho_{\vert{\cal H}'}, \frac{F'}{\vert F'\vert})$. Thus our original representation is equivalent to the finite dimensional Fredholm representation $(Ker F, \rho_{Ker F}, 0)$. It remains to show that 
every finite dimensional unitary representation of $G$ is trivial. 
If not, there is a finite dimensional unitary representation $\rho'$ of $G$ without fixed vectors. By Moore's ergodicity theorem the function 
$g\mapsto det(\rho'(g)), g\in G$, which is a homogeneous polynomial in the matrix coefficients of $\rho$, has to vanish at infinity. This contradicts the fact that this function is of absolute value 1 everywhere because $\rho'$ is unitary.
\end{proof21}

\begin{proof22} 
The same reasoning as in the proof of theorem 3.1 shows that the canonical homomorphism $R(\Gamma) \to {\cal KK}_\Gamma^{(p)}(\C,\C)$ is surjective and that\\
 ${\cal KK}^{(p)}_1((C^*_{max}(\Gamma),\C\Gamma),\C) = 0$. We want to show that the canonical map is injective as well. So let $[\rho], [\rho']\in R(\Gamma)$ be a virtual finite dimensional unitary representations of $\Gamma$ which define the same class in ${\cal KK}^{(p)}_0((C^*_{max}(\Gamma),\C\Gamma),\C)$.
According to 5.10 ii) one actually may find a homotopy of virtual unitary representations on a fixed finite dimensional $\Z/2\Z$-graded Hilbert space which connects $[\rho]$ and $[\rho']$. \\
Every higher rank group as considered in section 2 has Kazhdan's property T \cite{HV}, i.e. the trivial representation is an isolated point in the unitary dual with respect to the Fell topology. Property T is stable under taking finite products and passage to lattices \cite{HV}. Moreover it is easily seen that under the presence of property T not only the trivial representation but in fact every finite dimensional unitary representation is an isolated point in the unitary dual. Thus 
any homotopy of finite dimensional virtual unitary representations of a property T group is necessarily constant which proves our claim. 
\end{proof22}

\begin{proof23} 
Using 5.10 again we see (after using a smooth operator homotopy and the subtraction of a degenerate module) that every class in ${\cal KK}^{(p)}_*((C^*_{red}(\Gamma),\C\Gamma),\C)$ may be represented by a Fredholm module ${\cal E}\,=\,({\cal H}',\rho,0)$ over $C^*_{red}(\Gamma)$ with finite dimensional underlying Hilbert space ${\cal H}'$. According to \cite{Di}, 18.9.5 every finite dimensional representation of the reduced group $C^*$-algebra of a nonamenable locally compact group is identically zero. This applies in particular to the lattices under consideration which possess property T and are not compact and therefore nonamenable. Thus the considered Fredholm module $\cal E$ is identically zero
and the theorem follows.  
\end{proof23}

\end{document}